\documentclass[11pt,reqno]{amsart}
\usepackage{amsmath,mathrsfs}

\usepackage{enumerate}
\usepackage[english]{babel}
\usepackage{hyperref}
\usepackage[all,cmtip]{xy}
\entrymodifiers={+!!<0pt,\fontdimen22\textfont2>}

\theoremstyle{theorem}
\newtheorem{thm}{Theorem}[section]
\newtheorem{lem}[thm]{Lemma}
\newtheorem{prop}[thm]{Proposition}
\newtheorem{cor}[thm]{Corollary}
\newtheorem*{thm-intr}{Theorem}

\theoremstyle{definition}
\newtheorem{defn}[thm]{Definition}
\newtheorem{ex}[thm]{Example}
\newtheorem{ques}[thm]{Question}

\theoremstyle{remark}
\newtheorem{rem}[thm]{Remark}
\newtheorem*{ack}{Acknowledgements}

\def\opp{\mathrm{opp}}

\mathchardef\ordinarycolon\mathcode`\: 
\def\vcentcolon{\mathrel{\mathop\ordinarycolon}} 
\providecommand*\coloneqq{\mathrel{\vcentcolon\mkern-1.2mu}=}

\def\cast{$C^{*}$}
\def\M{{\mathbb M}} 
\DeclareMathOperator\id{id} 
\def\nsub{{$n$-sub\-homo\-gen\-eous}}
\begin{document}

\title{Approximations of Subhomogeneous Algebras}
\author{Tatiana Shulman and Otgonbayar Uuye}
\date{\today}   

\address{
Tatiana Shulman\\
Siena College\\
515 Loudon Road\\ 
Loudonville\\
NY 12211\\
USA
}
\address{
Otgonbayar Uuye\\
School of Mathematics\\
Cardiff University\\
Senghennydd Road\\
Cardiff, Wales, UK\\
CF24 4AG}

\begin{abstract} Let $n\ge 1$. Recall that a \cast-algebra is said to be \nsub\ if all its irreducible representations have dimension at most $n$. In this short note, we give various approximation properties characterising \nsub\- \cast-algebras.
\end{abstract}

\maketitle

\section{Introduction}
Let $A$ and $B$ be \cast-algebras and let $\phi\colon A \to B$ be a bounded linear map. For each integer $n \ge 1$, we can define maps 
	\begin{equation}
	\phi \otimes \id_{\M_{n}} \colon A\otimes \M_{n} \to B\otimes \M_{n},
	\end{equation}
where $\M_{n}$ denote the \cast-algebra of $n \times n$-matrices. We say that $\phi$ is {\em $n$-positive} if $\phi \otimes  \id_{\M_{n}}$ is positive and {\em $n$-contractive} if $\phi \otimes  \id_{\M_{n}}$ is contractive. We say that a map is {\em completely positive (completely contractive)} if it is $n$-positive ($n$-contractive) for all $n \ge 1$.  As usual, we abbreviate {\em completely positive} by c.p., {\em contractive and completely positive} by c.c.p., {\em unital and completely positive} by u.c.p.\ and {\em completely contractive} by c.c. Note that u.c.p.\ maps are c.c.p.\ and c.c.p.\ maps are c.c.\ by the Stinespring dilation theorem \cite[Theorem 1]{MR0069403}.

Finite-dimensional approximation properties of maps and \cast-algebras play an important role in the study of \cast-algebras. See \cite{MR2391387} for a comprehensive treatment.
\begin{defn}
A c.c.p.\ map $\theta\colon A \to B$ is said to be {\em nuclear} if there exist {\em finite-dimensional \cast-algebras} $F_{\alpha}$ and nets of c.c.p.\ maps $\phi_{\alpha}\colon A \to F_{\alpha}$ and $\psi_{\alpha}\colon F_{\alpha} \to B$ such that for all $x \in A$,
	\begin{equation}
	||(\theta - \psi_{\alpha} \circ \phi_{\alpha})(x)|| \to 0 \quad\text{as $\alpha \to \infty$}.
	\end{equation}
\end{defn}
	
\begin{defn}	
A \cast-algebra $A$ is said to be {\em nuclear} if the identity map $\id_{A}\colon A \to A$ is nuclear and {\em exact} if there exists a faithful representation $\pi\colon A \to B(H)$ which is nuclear.
\end{defn}

The following is the standard example.
\begin{ex} Let $\Gamma$ be a countable discrete group. Then the {\em reduced group \cast-algebra} $C^{*}_{\lambda}(\Gamma)$ is nuclear if and only if $\Gamma$ is amenable. In particular, the reduced group \cast-algebra $C^{*}_{\lambda}(F_{2})$ of a free group on two generators is non-nuclear. See \cite[Section 2.6.]{MR2391387}.
\end{ex}

It is well known that a \cast-algebra is nuclear if and only if the identity map is a point-norm limit of finite-rank c.c.p.\ maps. On the other hand, it was shown by De Canni\`ere and Haagerup that the identity map on $C^{*}_{\lambda}(F_{2})$ is a point-norm limit of finite-rank c.c.\ maps (cf.\  \cite[Corollary 3.11]{MR784292}). This is in contrast to the following theorem of Smith, which says that we recover nuclearity if we insist that the finite-rank c.c.\ maps to factor through finite-dimensional \cast-algebras.
\begin{thm}[{Smith \cite{MR813204}}] A \cast-algebra $A$ is nuclear if and only if there exist {\em finite-dimensional \cast-algebras} $F_{\alpha}$ and nets of c.c.\ maps $\phi_{\alpha}\colon A \to F_{\alpha}$ and $\psi_{\alpha}\colon F_{\alpha} \to A$ such that for all $x \in A$,
	\begin{equation}
	||(\id_{A} - \psi_{\alpha} \circ \phi_{\alpha})(x)|| \to 0 \quad\text{as $\alpha \to \infty$}.
	\end{equation}
\qed	
\end{thm}


All {\em abelian} \cast-algebras are nuclear. In fact, the standard proof based on partition of unities shows that one can take the finite-dimensional \cast-algebras $F_{\alpha}$'s to be abelian  and c.c.p.\ maps $\phi_{\alpha}$'s to be $*$-homomorphisms (cf.\ \cite[Proposition 2.4.2]{MR2391387}).

Our investigation grew out of the following simple question.
\begin{ques} Suppose that the there exist finite-dimensional  {\em abelian} \cast-algebras $F_{\alpha}$ and nets of c.c.p.\ maps $\phi_{\alpha}\colon A \to F_{\alpha}$ and $\psi_{\alpha}\colon F_{\alpha} \to A$ such that for all $x \in A$,
	\begin{equation}
	||(\id_{A} - \psi_{\alpha} \circ \phi_{\alpha})(x)|| \to 0 \quad\text{as $\alpha \to \infty$}.
	\end{equation}
Can we conclude that $A$ is abelian? Can we still conlcude that $A$ is abelian if we assume that the maps $\phi_{\alpha}$ and $\psi_{\alpha}$ are only c.c.?
\end{ques}

Not surprisingly, the answer is YES. In this paper we prove the following. 

Let $n \ge 1$. Recall that a \cast-algebra is said to be {\em \nsub} if all its irreducible representations have dimension $\le n$. Clearly, a \cast-algebra is abelian if and only if it is $1$-subhomogeneous. A finite-dimensional \cast-algebra is \nsub\- if and only if it is a finite product of matrix algebras $\M_{k}$ of size $k \le n$.

\begin{thm}\label{thm intr}
Let $A$ be a \cast-algebra and let $n \ge 1$ be an integer. Then the following are equivalent.
\begin{enumerate}
\item\label{it main nsub} The \cast-algebra $A$ is \nsub. 
\item\label{it main hom} There exist nets of {\em $*$-homomorphisms} $\phi_{\alpha}\colon A \to F_{\alpha}$ and c.c.p.\ maps $\psi_{\alpha}\colon F_{\alpha} \to A$, with $F_{\alpha}$ finite dimensional and \nsub, such that for all $x \in A$,
	\begin{equation}
	||x - \psi_{\alpha} \circ \phi_{\alpha}(x)|| \to 0 \quad\text{as $\alpha \to \infty$}.
	\end{equation}	
\item\label{it main ccp} There exist nets of c.c.p.\ maps $\phi_{\alpha}\colon A \to F_{\alpha}$ and $\psi_{\alpha}\colon F_{\alpha} \to A$, with $F_{\alpha}$ (finite dimensional and) \nsub, such that for all $x \in A$, 
	\begin{equation}
	||x - \psi_{\alpha} \circ \phi_{\alpha}(x)|| \to 0 \quad\text{as $\alpha \to \infty$}.
	\end{equation}
\item\label{it main cc} There exist nets of c.c.\ maps $\phi_{\alpha}\colon A \to F_{\alpha}$ and $\psi_{\alpha}\colon F_{\alpha} \to A$, with $F_{\alpha}$ (finite dimensional and) \nsub, such that for all $x \in A$, 
	\begin{equation}
	||x - \psi_{\alpha} \circ \phi_{\alpha}(x)|| \to 0 \quad\text{as $\alpha \to \infty$}.
	\end{equation} 
\end{enumerate}
\end{thm}
\begin{proof} The nontrivial implications are (\ref{it main nsub}) $\Rightarrow$ (\ref{it main hom}), (\ref{it main ccp}) $\Rightarrow$ (\ref{it main nsub}) and (\ref{it main cc}) $\Rightarrow$ (\ref{it main nsub}). See Theorem~\ref{thm intr 2} below.
\end{proof}

Our proof is based on the solution of the Choi conjecture \cite{MR0315460}, due to Tomiyama \cite{MR680854} and Smith \cite{MR686514}, and a contractive analogue of the Choi conjecture (see Theorem~\ref{thm choi c}). See also \cite{MR0397423} and \cite{MR715564}.

The following is a summary of the results.
\begin{thm}\label{thm intr 2} Let $A$ be a \cast-algebra and let $n \ge 1$ be an integer. Then the following are equivalent.
\begin{enumerate}
\item\label{it main2 nsub} The \cast-algebra $A$ is \nsub. 
\item\label{it main2 hom} There exist nets of {\em $*$-homomorphisms} $\phi_{\alpha}\colon A \to F_{\alpha}$ and c.c.p.\ maps $\psi_{\alpha}\colon F_{\alpha} \to A$, with $F_{\alpha}$ finite dimensional and \nsub, such that for all $x \in A$,
	\begin{equation}
	||x - \psi_{\alpha} \circ \phi_{\alpha}(x)|| \to 0 \quad\text{as $\alpha \to \infty$}.
	\end{equation}
\item\label{it main2 ccp} There exist nets of $n$-positive maps $\phi_{\alpha}\colon A \to F_{\alpha}$ and $\psi_{\alpha}\colon F_{\alpha} \to A$, with $F_{\alpha}$ finite dimensional and \nsub, such that for all $x \in A$, 
	\begin{equation}
	||x - \psi_{\alpha} \circ \phi_{\alpha}(x)|| \to 0 \quad\text{as $\alpha \to \infty$}.
	\end{equation}
\item\label{it main2 pos} All $n$-positive maps with domain and/or range $A$ are completely positive.
\item\label{it main2 cc} There exist  nets of $n$-contractive maps $\phi_{\alpha}\colon A \to F_{\alpha}$ and $(n+1)$-contractive maps $\psi_{\alpha}\colon F_{\alpha} \to A$, with $F_{\alpha}$ finite dimensional and \nsub, such that for all $x \in A$, 
	\begin{equation}
	||x - \psi_{\alpha} \circ \phi_{\alpha}(x)|| \to 0 \quad\text{as $\alpha \to \infty$}.
	\end{equation}
\item\label{it main2 contr} All $n$-contractive maps with range $A$ are completely contractive.
\end{enumerate}
\end{thm}
\begin{proof} See Theorems~\ref{thm app}, \ref{thm sub} and \ref{thm sub-contr}.
\end{proof}

In section~\ref{sec ab}, we show that even weaker approximation property characterises abelianness. See Theorem~\ref{thm takesaki proof}.

\begin{ack} The authors were partially supported by the Danish Research Council through the Centre for Symmetry and Deformation at the University of Copenhagen. The second author is an EPSRC fellow.
\end{ack}

\section{Subhomogeneous Algebras}\label{sec sub}

Let $n \ge 1$ be an integer.
\begin{defn} We say that a \cast-algebra is {\em \nsub} if all its irreducible representations have dimension $\le n$.
\end{defn}

\begin{ex}
A finite-dimensional \cast-algebra is \nsub\- if and only if it is a finite product of matrix algebras $\M_{k}$ with $k \le n$.
\end{ex}

In the following, we summarise some well-known properties of \nsub\- \cast-algebras. See also \cite[IV.1.4]{MR2188261}.

\begin{prop}\label{prop subhom} Let $n \ge 1$ be an integer. The following statements hold.
\begin{enumerate}
\item\label{i subalg} A \cast-subalgebra of a \nsub\- algebra is \nsub.
\item\label{i ab} A \cast-algebra $A$ is \nsub\- if and only if $A \subseteq \M_{n}(B)$ for some abelian \cast-algebra $B$.   
\item\label{i dd} A \cast-algebra $A$ is \nsub\- if and only if its bidual $A^{**}$ is \nsub\- as a \cast-algebra.
\item\label{i prod} The product/sum of \cast-algebras $A_{i}$, $i \in I$, is \nsub\- if and only if each $A_{i}$, $i \in I$, is \nsub. 
\item\label{i ext} Let $0 \to I \to A \to B \to 0$ be an extension of \cast-algebras. Then $A$ is \nsub\- if and only if $I$ and $B$ are \nsub.
\end{enumerate}
\end{prop}
\begin{proof}
\begin{enumerate}
\item[(\ref{i subalg})] Follows from \cite[Proposition 4.1.8]{MR548006}. 
\item[(\ref{i ab})] If $A$ is \nsub, then $A \subseteq \M_{n}(l^{\infty}(\widehat{A}))$, where $\widehat{A}$ denote the set of unitary equivalence classes of irreducible representations of $A$. The other direction follows from (\ref{i subalg}).
\item[(\ref{i dd})] Since $A \subseteq A^{**}$, if $A^{**}$ is \nsub, then so is $A$ by (\ref{i subalg}). Conversely, if $A$ is \nsub, then writing $A \subseteq \M_{n}(B)$ with $B$ abelian using (\ref{i ab}), we see that $A^{**} \subseteq \M_{n}(B^{**})$. We are done by (\ref{i ab}), since $B^{**}$ is abelian.
\item[(\ref{i prod})]  Follows from (\ref{i ab}).
\item[(\ref{i ext})] Follows from (\ref{i dd}) and (\ref{i prod}), since $A^{**} \cong I^{**} \oplus B^{**}$.
\end{enumerate}
\end{proof}

The structure of \nsub\- \cast-algebras can be rather complicated (see for instance \cite{MR0139025}). However, the situation for von Neumann algebras is well-known to be very simple. 
\begin{lem}\label{lem vN} Suppose that a von Neumann algebra $M$ is \nsub\- as a \cast-algebra. Then
	\begin{equation}
	M \cong \prod_{k \le n} \M_{k}(B_{k}),
	\end{equation}
where $B_{k}$, $k \le n$, are {\em abelian} von Neumann algebras.	
\end{lem}
\begin{proof} Since exactness passes to \cast-subalgebras, \nsub\- algebras are exact by Proposition~\ref{prop subhom}(\ref{i ab}). Now \cite[Proposition 2.4.9]{MR2391387} completes the proof.
\end{proof}

Subhomogeneous algebras are type I, hence nuclear (cf.\ \cite[Proposition 2.7.7]{MR2391387}). Scrutinizing the proof, we see that the following slightly stronger approximation property holds. We consider the unital case first.
\begin{thm}\label{thm app unital} Let $n \ge 1$ and let $A$ be a unital \nsub\- \cast-algebra. Then there exist finite-dimensional \nsub\- \cast-algebras $F_{\alpha}$ and nets of {unital \em $*$-homomorphisms} $\phi_{\alpha}\colon A \to F_{\alpha}$ and u.c.p.\ maps $\psi_{\alpha}\colon F_{\alpha} \to A$ such that for all $x \in A$,
	\begin{equation}
	||x - \psi_{\alpha} \circ \phi_{\alpha}(x)|| \to 0 \quad\text{as $\alpha \to \infty$}.
	\end{equation}
\end{thm}

\begin{defn} Let $A$ and $B$ be unital \cast-algebras. We say a u.c.p.\ map $\theta\colon A \to B$ is {\em $n$-factorable} if it can be expressed as a composition $\theta = \psi \circ \phi$, where $\phi\colon A \to F$ is a unital $*$-homomorphism and $\psi\colon F \to B$ a u.c.p.\ map and $F$ a finite-dimensional \nsub\- \cast-algebra.

\end{defn}
\begin{lem}\label{lem convex} For any unital \cast-algebras $A$ and $B$, the set of $n$-factorable maps $A \to B$ is convex.
\end{lem}
\begin{proof} The proof of \cite[Lemma 2.3.6]{MR2391387} applies.
\end{proof}

\begin{lem}\label{lem app} Let $F$ be a finite-dimensional \cast-algebra and let $A$ be a unital \cast-algebra. Then u.c.p.\ maps $F \to A^{**}$ can be approximated by u.c.p.\ maps $F \to A$ in the point-ultraweak topology. 
\end{lem}
\begin{proof} We claim that c.p.\ maps $F \to A$ correspond bijectively to positive elements in $F \otimes A$. Indeed, for matrix algebras this is a well-known result of Arveson (cf.\ \cite[Proposition 1.5.12]{MR2391387}). The general case follows, since $F$ is a finite product of matrix algebras and for c.p.\ maps finite products and finite coproducts coincide. Since positive elements in $F \otimes A$ is ultraweakly dense in the positive elements in $F \otimes A^{**} \cong (F \otimes A)^{**}$, we see that c.p.\ maps $F \to A^{**}$ can be approximated by c.p.\ maps $F \to A$ in the point-ultraweak topology.

Let $\psi\colon F \to A^{**}$ be a u.c.p.\ map and let $\psi_{\lambda}\colon F \to A$ be a net of c.p.\ maps converging to $\psi$ in the point-ultraweak topology. Since $\psi_{\lambda}(1_{F}) \in A$ is a net converging to $1_{A}$ weakly, by passing to convex linear combinations, we may assume that $\psi_{\lambda}(1_{F})$ converges to $1_{A}$ in norm and passing to a subnet we may assume that $\psi_{\lambda}(1_{F})$ is invertible. Then $\tilde\psi_{\lambda}(x) \coloneqq \psi_{\lambda}(1_{F})^{-1/2}\psi_{\lambda}(x)\psi_{\lambda}(1_{F})^{-1/2}$, $x \in F$, gives the required approximation. 
\end{proof}

\begin{proof}[Proof of Theorem~\ref{thm app unital}]
For $n = 1$ and $A$ unital abelian, the claim follows from the classical proof of nuclearity for abelian algebras (cf.\ \cite[Proposition 2.4.2]{MR2391387}). 

For general $n$, first assume that $A$ is of the form 
	\begin{equation}\label{i elementary}
	\prod_{k \le n} \M_{k}(A_{k}),
	\end{equation}
where $A_{k}$, $k \le n$, are unital abelian \cast-algebras. Then the claim is easily deduced from the case $n = 1$.

Now we consider a general \nsub\- $A$. By Proposition~\ref{prop subhom}(\ref{i dd}) and Lemma~\ref{lem vN}, the bidual $A^{**}$ is of the form (\ref{i elementary}), hence $\id_{A^{**}}$ can be approximated by $n$-factorable maps $A^{**} \to A^{**}$ in point-norm topology. Then by Lemma~\ref{lem app}, $\id_{A}$ can be approximated by $n$-factorable maps in point-weak topology. Now  Lemma~\ref{lem convex} and \cite[Lemma 2.3.4]{MR2391387} completes the proof.
\end{proof}

As a corollary we obtain the following.
\begin{thm}\label{thm app} Let $A$ be a \cast-algebra and let $n \ge 1$ be an integer. Then $A$ is \nsub\- if and only if there exist nets of {\em $*$-homomorphisms} $\phi_{\alpha}\colon A \to F_{\alpha}$ and c.c.p.\ maps $\psi_{\alpha}\colon F_{\alpha} \to A$, with $F_{\alpha}$  finite-dimensional \nsub, such that for all $x \in A$,
	\begin{equation}
	||x - \psi_{\alpha} \circ \phi_{\alpha}(x)|| \to 0 \quad\text{as $\alpha \to \infty$}.
	\end{equation}
\end{thm}
\begin{proof}
%

($\Rightarrow$) The unitization $A^{+}$ is \nsub, hence $\id_{A^{+}}$ can be approximated as in Theorem~\ref{thm app unital}. Now restrict $\phi_{\alpha}$ to $A$ and replace $\psi_{\alpha}$ by $e_{\beta}\psi_{\alpha}e_{\beta}$, where $e_{\beta}$ is an approximate
unit in $A$ (cf.\ \cite[Exercise 2.3.4]{MR2391387}).

($\Leftarrow$) Clearly $A$ is a \cast-subalgebra of $\prod_{\alpha} F_{\alpha}$. By Proposition~\ref{prop subhom}(\ref{i prod} \& \ref{i subalg}), $A$ is \nsub.
\end{proof}

It turns out that much weaker approximation properties imply $n$-sub\-homogeneity. Our first result depends on the following.
\begin{thm}[{Choi, Tomyama, Smith}]\label{thm choi} Let $A$ and $B$ be \cast-algebras and let $n \ge 1$ be an integer. Then all $n$-positive maps $A \to B$ are completely positive if and only if $A$ or $B$ is \nsub.
\end{thm}
\begin{proof} Choi proved the sufficiency ($\Leftarrow$) for $A = \M_{n}(D)$ (cf.\ \cite[Theorem 8]{MR0315460}) and $B = \M_{n}(D)$ (cf.\ \cite[Theorem 7]{MR0315460}) with $D$ abelian and conjectured the necessity ($\Rightarrow$). A complete proof is obtained by Tomiyama (cf.\ \cite[Theorem 1.2]{MR680854}). The necessity was also proved by Smith (cf.\ \cite[Theorem 3.1]{MR686514}).
\end{proof}

\begin{thm}\label{thm sub} Let $A$ be a \cast-algebra and let $n \ge 1$ be an integer. Then the following are equivalent.
\begin{enumerate}
\item\label{item app} There exist nets  of $n$-positive maps $\phi_{\alpha}\colon A \to F_{\alpha}$ and $\psi_{\alpha}\colon F_{\alpha} \to A$, with $F_{\alpha}$ finite-dimensional \nsub, such that for all $x \in A$, 
	\begin{equation}
	||x - \psi_{\alpha} \circ \phi_{\alpha}(x)|| \to 0 \quad\text{as $\alpha \to \infty$}.
	\end{equation}
\item\label{item domain} All $n$-positive maps with domain $A$ are completely positive.
\item\label{item range} All $n$-positive maps with range $A$ are completely positive.
\item\label{item dom-range} All $n$-positive maps $A \to A$ are completely positive.
\item\label{item repn} The \cast-algebra $A$ is \nsub.
\end{enumerate}
\end{thm}
\begin{proof} Let $\phi_{\alpha}\colon A \to F_{\alpha}$ and $\psi_{\alpha}\colon F_{\alpha} \to A$ be a $n$-positive approximation of $\id_{A}$ in point-norm topology, with $F_{\alpha}$ (finite-dimensional) \nsub. Let $\theta\colon A \to B$ be a $n$-positive map. Then $\theta \circ \psi_{\alpha} \colon F_{\alpha} \to B$ is a $n$-positive map with \nsub\- domain, hence c.p.\ map by Theorem~\ref{thm choi} and $\phi_{\alpha}\colon A \to F_{\alpha}$ is a $n$-positive map with \nsub\- range, hence also c.p. Since $\theta$ is the point-norm limit of $(\theta \circ \psi_{\alpha}) \circ \phi_{\alpha}$, we see that $\theta$ is c.p. Hence (\ref{item app}) $\Rightarrow$ (\ref{item domain}). Similarly (\ref{item app}) $\Rightarrow$ (\ref{item range}).

The implications (\ref{item domain}) $\Rightarrow$ (\ref{item dom-range}) and  (\ref{item range}) $\Rightarrow$ (\ref{item dom-range}) are clear and the implication (\ref{item dom-range}) $\Rightarrow$ (\ref{item repn}) is immediate from Theorem~\ref{thm choi}. Finally, the implication (\ref{item repn}) $\Rightarrow$ (\ref{item app}) follows from Theorem~\ref{thm app}. 
\end{proof}

\begin{rem} The sufficiency in Theorem~\ref{thm choi} can be deduced from the cases $A = \M_{n}$ (cf.\ \cite[Theorem 6]{MR0315460}) and $B = \M_{n}$ (cf.\ \cite[Theorem 5]{MR0315460}) using Theorem~\ref{thm sub}.
\end{rem}
Now we consider the contractive analogue.
\begin{lem} Let $\tau_{n}\colon \M_{n} \to \M_{n}$, $n \ge 1$, denote the transpose map and let $m \ge 1$. Then 
	\begin{equation}
	||\tau_{n} \otimes \id_{\M_{m}}\colon \M_{n} \otimes \M_{m} \to \M_{n} \otimes \M_{m}|| = \min\{m, n\}.
	\end{equation}
\end{lem}
\begin{proof} For $n \le m$, this is well-known. The general case follows from the identity 
	\begin{equation}
	(\tau_{n} \otimes \tau_{m}) \circ (\tau_{n} \otimes \id_{\M_{m}}) = \id_{\M_{n}} \otimes \tau_{m},
	\end{equation}
since $\tau_{n} \otimes \tau_{m}$ can be identified with $\tau_{nm}$, hence an isometry.	
\end{proof}

\begin{cor} Let $n \ge 2$ be an integer. Then the map
	\begin{equation}
	\frac{1}{n-1}\tau_{n}\colon \M_{n} \to \M_{n}
	\end{equation}
is $(n-1)$-contractive, but not $n$-contractive.	
\end{cor}

As a corollary, we obtain the following contractive analogue of Theorem~\ref{thm choi}. Note that we have only one of the directions (cf.\ \cite[Theorem C]{MR0397423}).
\begin{thm}\label{thm choi c} Let $A$ and $B$ be \cast-algebras and let $n \ge 1$ be an integer. If $A$ and $B$ both admit irreducible representations of dimension $\ge (n+1)$, then there exists an $n$-contractive map $A \to B$ which is not $(n+1)$-contractive.
\end{thm}
\begin{proof} The proof of \cite[Theorem 3.1]{MR686514} applies. See also \cite[Lemma 1.1 \& Theorem 1.2]{MR680854}
\end{proof}

\begin{thm}\label{thm sub-contr}
Let $A$ be a \cast-algebra and let $n \ge 1$ be an integer. Then the following are equivalent.
\begin{enumerate}
\item\label{item app c} There exist nets of $n$-contractive maps $\phi_{\alpha}\colon A \to F_{\alpha}$ and $(n+1)$-contractive maps $\psi_{\alpha}\colon F_{\alpha} \to A$, with $F_{\alpha}$ finite-dimensional \nsub, such that for all $x \in A$, 
	\begin{equation}
	||x - \psi_{\alpha} \circ \phi_{\alpha}(x)|| \to 0 \quad\text{as $\alpha \to \infty$}.
	\end{equation}
\item\label{item range +} All $n$-contractive maps with range $A$ are $(n+1)$-contractive.
\item\label{item range c} All $n$-contractive maps with range $A$ are completely contractive.
\item\label{item dom-range +} All $n$-contractive maps $A \to A$ are $(n+1)$-contractive.
\item\label{item dom-range c} All $n$-contractive maps $A \to A$ are completely contractive.
\item\label{item repn c} The \cast-algebra $A$ is \nsub.
\end{enumerate}
\end{thm}
\begin{proof} We prove the implications
	\begin{equation}
	\xymatrix{
	(\ref{item app c})\ar@{=>}[d] & (\ref{item repn c}) \ar@{=>}[l] \ar@{=>}[r] \ar@{=>}[dl]& (\ref{item range c})\ar@{=>}[d]\\
	(\ref{item range +}) \ar@{=>}[r]& (\ref{item dom-range +})\ar@{=>}[u] & (\ref{item dom-range c})\ar@{=>}[l]
	}.
	\end{equation}

The implications (\ref{item range c}) $\Rightarrow$ (\ref{item dom-range c}), (\ref{item dom-range c}) $\Rightarrow$ (\ref{item dom-range +}) and (\ref{item range +}) $\Rightarrow$ (\ref{item dom-range +}) are clear. The implication (\ref{item dom-range +}) $\Rightarrow$ (\ref{item repn c}) follows from Theorem~\ref{thm choi c} and the implication (\ref{item repn c}) $\Rightarrow$ (\ref{item app c}) follows from Theorem~\ref{thm app}. The implication (\ref{item repn c}) $\Rightarrow$ (\ref{item range c}) follows from \cite[Theorem 2.10]{MR686514}. Since (\ref{item range c}) $\Rightarrow$ (\ref{item range +}) is clear, we also have (\ref{item repn c}) $\Rightarrow$ (\ref{item range +}).

Finally, the implication (\ref{item app c}) $\Rightarrow$ (\ref{item range +}) is analogous to the proof of Theorem~\ref{thm sub}((\ref{item app}) $\Rightarrow$ (\ref{item range})).
\end{proof}

Compare with the Loebl conjecture \cite{MR0397423}, solved affirmatively by Huruya-Tomiyama \cite{MR715564} and Smith \cite{MR686514}. 
\begin{rem}\label{rem loebl} Note that the statement
\begin{enumerate}
\item[(7)] All $n$-contractive maps with {\em domain} $A$ are $(n+1)$-contractive.
\end{enumerate}
is {\em not} equivalent to the conditions in Theorem~\ref{thm sub-contr} in general (cf.\ \cite[Theorem C]{MR0397423}).
\end{rem}
\section{The Abelian Case}\label{sec ab}

Specialising to $n = 1$ in Theorem~\ref{thm sub}, we obtain the following. \begin{thm}\label{thm cp} 
Let $A$ be a \cast-algebra. Suppose that there exist nets of contractive positive maps $\phi_{\alpha}\colon A \to F_{\alpha}$ and $\psi_{\alpha}\colon F_{\alpha} \to A$, with $F_{\alpha}$ {\em abelian}, such that for all $x \in A$,
	\begin{equation}
	||x - \psi_{\alpha} \circ \phi_{\alpha}(x)|| \to 0 \quad\text{as $\alpha \to \infty$}.
	\end{equation}	
Then $A$ is abelian. 
\end{thm}
We give an alternative proof.

\begin{proof} First note that $\phi_{\alpha}$ and $\psi_{\alpha}$ are c.c.p.\ (cf.\ \cite[Theorem 3 \& 4]{MR0069403}).

Unitizing if necessary, we may assume that $A$ is unital. Let $A^{\opp}$ denote the opposite algebra of $A$. Then the canonical map $\iota\colon A \to A^{\opp}$ is a pointwise limit of c.c.p.\ maps $\psi_{\alpha}^{\opp}\circ\phi_{\alpha}\colon A \to F_{\alpha} \cong F_{\alpha}^{\opp} \to A^{\opp}$, hence a c.c.p.\ map. Moreover, since $\iota$ sends unitaries to unitaries, its multiplicative domain is the whole of $A$. It follows that $\iota$ is a $*$-homomorphism and $A$ is abelian.\footnote{Walter's $3 \times 3$ trick shows that if the canonical map $\iota\colon A \to A^{\opp}$ is $3$-positive then $A$ is abelian (cf.\ \cite{MR1963763}). 
}
\end{proof}

In fact, the following is true.
\begin{thm}\label{thm takesaki proof} Let $\theta\colon A\to B$ be an {\em injective} $*$-homomorphism. Suppose that there exist nets of contractive maps $\phi_{\alpha}\colon A \to F_{\alpha}$ and $2$-contractive maps  $\psi_{\alpha}\colon F_{\alpha} \to B$, with $F_{\alpha}$ {\em abelian}, such that for all $x \in A$,
	\begin{equation}\label{2-limit}
	||(\theta - \psi_{\alpha} \circ \phi_{\alpha})(x)|| \to 0 \quad\text{as $\alpha \to \infty$}.
	\end{equation}
Then $A$ is abelian. 
\end{thm}

Our main tool is the following beautiful theorem of Takesaki. Let $A_{1}$ and $A_{2}$ be \cast-algebras. The {\em injective cross-norm} of $A_{1}$ and $A_{2}$ is defined by
	\begin{equation}
	||x||_{\lambda} \coloneqq \sup|(\varphi_{1} \otimes \varphi_{2})(x)|
	\end{equation}
where $\varphi_{1}$ and $\varphi_{2}$ run over all contractive linear functionals of $A_{1}$ and $A_{2}$, respectively. The {\em injective \cast-cross-norm} of $A_{1}$ and $A_{2}$ is defined by
	\begin{equation}
	||x||_{\min} \coloneqq \sup||(\pi_{1} \otimes \pi_{2})(x)||
	\end{equation}
where $\pi_{1}$ and $\pi_{2}$ run over all unitary representations of $A_{1}$ and $A_{2}$, respectively.

Note that we always have $||\cdot||_{\lambda} \le ||\cdot||_{\min}$ (cf.\ \cite[IV.4(12)]{MR1873025}).
	
\begin{thm}[{Takesaki \cite[Theorem IV.4.14]{MR1873025}}]\label{thm takesaki}
Let $A_{1}$ and $A_{2}$ be \cast-algebras. Then the norms $|| \cdot ||_{\min}$ and $|| \cdot ||_{\lambda}$ on $A_{1} \otimes A_{2}$ are equal if and only if $A_{1}$ or $A_{2}$ is abelian.
\qed
\end{thm}

Equipped with Takesaki's theorem, we can now mimic the proof that nuclear \cast-algebras are tensor-nuclear\footnote{It is actually closer to the proof of the fact that exact \cast-algebras with Lance's weak expectation property are tensor-nuclear.} (cf.\ \cite[Proposition 3.6.12]{MR2391387}).
\begin{proof}[Proof of Theorem~\ref{thm takesaki proof}]
We show that   for any $x \in A \otimes \M_{2}$, we have $||x||_{\min} \le ||x||_{\lambda}$. Then Theorem~\ref{thm takesaki} completes the proof. 


Let $x \in A \otimes \M_{2}$. The map 
	\begin{equation}
	\theta \otimes_{\min} \id_{\M_{2}} \colon A \otimes_{\min} \M_{2} \to B \otimes_{\min} \M_{2}
	\end{equation} 
is an injective $*$-homomorphism, hence an isometry. Thus
	\begin{equation}
	||x||_{A \otimes_{\min} \M_{2}} =||\theta \otimes \id_{\M_{2}} (x)||_{B \otimes_{\min} \M_{2}}.
	\end{equation}
Writing $x$ as the sum of elementary tensors, we see that
	\begin{equation}
	||(\theta - \psi_{\alpha} \circ \phi_{\alpha}) \otimes \id_{\M_{2}} (x)||_{B \otimes_{\min} \M_{2}} \to 0 \quad\text{as $\alpha \to \infty$}.
	\end{equation}
Hence 
	\begin{equation}
	||x||_{A \otimes_{\min} \M_{2}} = \lim_{n\to \infty}||(\psi_{\alpha} \circ \phi_{\alpha}) \otimes \id_{\M_{2}} (x)||_{B \otimes_{\min} \M_{2}}. 
	\end{equation}	
On the other hand, it follows from the assumptions that the maps
	\begin{align}
	\phi_{\alpha} \otimes_{\lambda} \id_{\M_{2}}&\colon A \otimes_{\lambda} \M_{2} \to F_{\alpha} \otimes_{\lambda} \M_{2}\quad\text{and}\\
	\psi_{\alpha} \otimes_{\min} \id_{\M_{2}}&\colon F_{\alpha} \otimes_{\min} \M_{2} \to B \otimes_{\min} \M_{2}
	\end{align} 
are contractions and since $F_{\alpha}$ is abelian, the canonical map 
	\begin{equation}
	F_{\alpha} \otimes_{\min} \M_{2} \to F_{\alpha} \otimes_{\lambda} \M_{2}
	\end{equation}
is an isometry by Theorem~\ref{thm takesaki}. Hence, we have
	\begin{align}
	||(\psi_{\alpha} \circ\phi_{\alpha}) \otimes \id_{\M_{2}}(x)||_{B \otimes_{\min} \M_{2}} &\le ||\phi_{\alpha} \otimes \id_{\M_{2}}(x)||_{F_{\alpha} \otimes_{\min} \M_{2}}\\
	 &= ||\phi_{\alpha} \otimes \id_{\M_{2}}(x)||_{F_{\alpha} \otimes_{\lambda} \M_{2}}\\
	&\le ||x||_{A \otimes_{\lambda} \M_{2}}.
	\end{align}
It follows that $||x||_{\min} \le ||x||_{\lambda}$. 
\end{proof}

%
%
%
 

\bibliographystyle{amsalpha}
\bibliography{../BibTeX/biblio}

\end{document}